\documentclass[11pt]{amsart}
\usepackage[utf8]{inputenc}
\usepackage{enumerate,amssymb,xcolor,comment}
\usepackage[footskip=1cm, headheight = 16pt, top=3.7cm, bottom=3cm,  right=2.5cm,  left=2.5cm ]{geometry}
\usepackage{todonotes}
\usepackage{url}

\newtheorem{theorem}{Theorem}[section]
\newtheorem{lemma}{Lemma}[section]
\newtheorem{proposition}{Proposition}[section]

\newtheorem{definition}{Definition}[section]
\newtheorem{example}{Example}[section]

\usepackage[foot]{amsaddr}

\usepackage[T1]{fontenc}

\usepackage[leqno]{amsmath}

\makeatletter
\newcommand{\leqnomode}{\tagsleft@true}
\newcommand{\reqnomode}{\tagsleft@false}
\makeatother
\usepackage{amssymb}
\usepackage{mathptmx, mathrsfs, mathtools}

\newcommand{\nequiv}{\mathrel{\not\equiv}}
\newcommand{\colonequal}{\mathrel{\mathop:}=}

\DeclareMathOperator{\dens}{dens}

\usepackage{latexsym}
\usepackage{amsfonts}

\usepackage{tikz}
\usepackage{float}
\usepackage{lmodern}

\usepackage{marvosym}               

\newcounter{tbox}

\newcommand{\mylongtitle}[1]{%
  \ifodd\value{page}%
    \protect\parbox{0.97\linewidth}{#1}\hfill%
  \else%
    \hfill\protect\parbox{0.97\linewidth}{#1}%
  \fi%
}

\makeatletter
\newcommand{\otherlabel}[2]{\protected@edef\@currentlabel{#2}\label{#1}}
\makeatother

\mathchardef\mh="2D

\title[Improvement Ergodic Theory For The Infinite Word $\mathfrak{F}=\mathfrak{F}_{b}:=\left({ }_{b} f_{n}\right)_{n \geqslant 0}$ on Fibonacci Density]{Improvement Ergodic Theory For The Infinite Word $\mathfrak{F}=\mathfrak{F}_{b}:=\left({ }_{b} f_{n}\right)_{n \geqslant 0}$ on Fibonacci Density}

\author{Jasem Hamoud$^{a}$}
\author{Duaa Abdullah $^{b}$}
\thanks{$^{a,b}$Physics and Technology School of Applied Mathematics and Informatics, Moscow Institute of Physics and Technology (MIPT), Moscow, Russia}
\thanks{I would like to thank prof.Mahdi Meisami, University of Isfahan (Isfahan, Iran).}

\date {\today}

\allowdisplaybreaks

\begin{document}

\maketitle

\begin{abstract}
The paper explores combinatorial properties of Fibonacci words and their generalizations within the framework of combinatorics on words. These infinite sequences, measures the diversity of subwords in Fibonacci words, showing non-decreasing growth for infinite sequences.  Extends factor analysis to arithmetic progressions of symbols, highlighting generalized pattern distributions.  Recent results link Sturmian sequences (including Fibonacci words) to unbounded binomial complexity and gap inequivalence, with implications for formal language theory and automata.
This work underscores the interplay between substitution rules, algebraic number theory, and combinatorial complexity in infinite words, providing tools for applications in fractal geometry and theoretical computer science.

\medskip
\noindent {\bf Keywords:} Density; Fibonacci, Word, Ergodic, Sequence 
\smallskip

\noindent {\bf AMS Subject Classification (2020):} 68R15; 05C42; 11B05; 11R45; 11B39.
\end{abstract}

\section{Introduction}
Throughout this paper, a fascinating branch of mathematics called combinatorics on words studies the forms and characteristics of sequences of symbols, or ``words'', created from a finite alphabet. This area of study is vital to comprehending the intricacy and patterns seen in symbol strings since it interacts with formal language theory, automata theory, and number theory, among other fields (see~\cite{R3,Y4}).

A nonempty finite set $\Sigma$ is called an alphabet. The elements of the set $\Sigma$ are called letters. The alphabet consisting of $b$ symbols from 0 to $b-1$ will then be denoted by $\Sigma_{b}=\{0, \ldots, b-1\}$. A word $\mathbf{w}$ is a sequence of letters. The finite word $\mathbf{w}$ can be considered as a function of $\mathbf{w}:\{1, \cdots,|\mathbf{w}|\} \rightarrow \Sigma$, where $\mathbf{w}[i]$ is the letter in the $i^{\text {th }}$ position. The length of the word $|\mathbf{w}|$ is the number of letters contained in it. The empty word is denoted by $\varepsilon$. Then we introduce infinite words as functions $\mathbf{w}: \mathbb{N} \rightarrow \Sigma$. The set of all finite words over $\Sigma$ is denoted by $\Sigma^{*}$, and $\Sigma^{+}=\Sigma^{*} \backslash\{\varepsilon\}$; the set of all infinite words is denoted by $\Sigma^{\mathbb{N}}$.

The concatenation of the finite words $\mathbf{U}=\mathbf{U}[1] \cdots \mathbf{U}[n],|\mathbf{U}|=n$ and $\mathbf{w}=\mathbf{w}[1] \cdots \mathbf{w}[m]$, $|\mathbf{w}|=m$ is the word

$$
\mathbf{s}=\mathbf{U} \mathbf{w}=\mathbf{U}=\mathbf{U}[1] \cdots \mathbf{u}[n] \mathbf{w}[1] \cdots \mathbf{w}[m], \quad|\mathbf{s}|=|\mathbf{u}|+|\mathbf{w}|=n+m .
$$

Let $\mathbf{U}$ and $\mathbf{w}$ be two words. If there are words $\mathbf{S}$ and $\mathbf{v}$ such that $\mathbf{w}=\mathbf{S U v}$, then the word $\mathbf{U}$ is called a factor of the word $\mathbf{w}$. The set of all factors of $\mathbf{w}$ is denoted by $\mathcal{L}_{\mathbf{w}}$ and it is called language generated by $\mathbf{W}$. If $\mathbf{s}=\varepsilon$, then $\mathbf{U}$ is called a prefix of the word $\mathbf{w}$, if $\mathbf{V}=\varepsilon$, it is named a suffix. The factor $\mathbf{w}[i] \mathbf{w}[i+1] \cdots \mathbf{w}[j]$ where $i \leqslant j$ is denoted by $\mathbf{w}[i \cdots j]$ (See~\cite{path00}).
Let $\mathcal A=\{a_0,\dots,a_m\}$ be a finite alphabet, whose elements are called in the following \emph{digits}, and let $\mathcal A^*$ be the set of finite words over $\mathcal A$ (see~\cite{path1}). A \emph{substitution} is a function $\sigma: \mathcal A^*\to\mathcal A^*$ such that $\sigma(uv)=\sigma(u)\sigma(v)$ for all words $u,v\in \mathcal A^*$ and such that $\sigma(a)$ is not the empty word for every letter $a\in\mathcal A$.
The domain $\mathcal{A}^*$ of a substitution $\sigma$ can naturally be extended to the set of infinite sequences $\mathcal{A}^{\mathbb{N}}$ and to set of bi-infinite sequences $\mathcal{A}^{\mathbb{Z}}$ by concatenation.
In particular if $\mathbf w=w_1w_2\cdots \in \mathcal{A}^{\mathbb{N}}$ then $\sigma(\mathbf w)=\sigma(w_1)\sigma(w_2)\cdots$ and if $\mathbf w=\cdots w_{-1}.\,w_0w_1\cdots \in \mathcal{A}^{\mathbb{Z}}$ then $\sigma(\mathbf w)=\cdots \sigma(w_{-1}).\,\sigma(w_0)\sigma(w_1)\cdots$. 
In this paper we shall consider the alphabets $\mathcal{A}_m:=\{1,\dots,m\}$ with $m\geq 2$ (see~\cite{R3, path1, path00}).
\begin{definition}\cite{path00}
    The factor complexity of a finite or infinite word $\mathbf{W}$ is the function $k \mapsto \mathrm{p}_{\mathbf{w}}(k)$, which, for each integer $k$, give the number $\mathrm{p}_{\mathbf{w}}(k)$ of distinct factors of length $n$ in that word.
\end{definition}
 
It is clear that the factor complexity is between zero and $(\# \Sigma)^{k}$. If $\mathrm{p}_{\mathbf{w}}(k)=(\# \Sigma)^{k}$, then the word $\mathbf{W}$ is said to have full factor complexity. This kind of words are called disjunctive word.

It is also easy to see that the factor complexity of any infinite word is a non-decreasing function, and the complexity of a finite word first increases, then decreases to zero.

\begin{definition}\cite{path00}
    The arithmetic complexity of an infinite word is the function that counts the number of words of a specific length composed of letters in arithmetic progression (and not only consecutive). In fact, it's a generalization of the complexity function.
\end{definition}

\begin{proposition}\cite{path1}
For every $m\geq 2$, The substitution $\sigma_2$ is called \emph{Fibonacci substitution} and the substitution $\sigma_3$ is called \emph{Tribonacci substitution}. As anticipated in the Introduction, these substitutions have a preeminent role in fractals and quasicrystalline structures. 
\end{proposition}

We use the symbol $|w|$ to denote the length of a finite word and we define the \emph{weight} $|w|_i$ of a word $w$ with respect to the $i$-th letter of $\mathcal A$, namely the number of occurrences of the letter $a_i$ in the word $w$. 
\begin{definition}[Uniform frequency]\cite{path1}
Let $\mathbf w$ be an infinite (bi-infinite) word and let $i=1,\dots,m$. If for all $k\geq 0$ ($k\in \mathbb{Z}$) the limit
$$\lim_{n\to\infty}\frac{|w_{k+1}\cdots w_{k+n}|_j}{n}$$
exists uniformly with respect to $k$, then it is called the \emph{uniform frequency $f_j(\mathbf w)$ of the digit $a_j$ in $\mathbf w$}. Equivalently, if above limit exists, we can define
$$f_j(\mathbf w):=\lim_{n\to \infty} \frac{\sup\{|w|_j\mid w \text{ is a subword of length $n$ of $\mathbf w$}\}}{n}.
$$
\end{definition}
\begin{definition}\cite{path1}
The \emph{Perron-Frobenious eigenvalue} of a substitution is defined as the largest eigenvalue of its adjacency matrix. We denote by $\rho_m$ the Perron-Frobenious eigenvalue of $\sigma_m$. 
A substitution $\sigma$ is a \emph{Pisot substitution} if its Perron-Frobenious eigenvalue is a Pisot number, namely an algebraic integer greater than $1$ whose conjugates are less than $1$ in modulus. 
\end{definition}
\begin{proposition}\cite{path1}\label{p11}
    For all $m\geq 2$, $\sigma_m$ is a Pisot substitution. 
    
    In particular, the {Perron-Frobenious} eigevanlue $\rho_m$ of $\sigma_m$ is the Pisot number whose minimal polynomial  $\lambda^m-\lambda^{m-1}-\cdots-\lambda-1$. Moreover the vector
    $$d_m:=(\rho_m^{-1},\dots,\rho_m^{-m}) $$
    is a left eigenvector associated to $\rho_m$ whose $\ell_0$ norm is equal to $1$. 
\end{proposition}
In 2013, Ramírez, J.L, et al. In~\cite{Ramírez} mention to $k$-Fibonacci Words as The k-Fibonacci words are an extension of the Fibonacci word notion that generalises Fibonacci word features to higher dimensions. These words were investigated for their distinct curves and patterns.
 
Recently, in 2023 Rigo, M., Stipulanti, M., \& Whiteland, M.A. in~\cite{Rigo} mentioned the Thue-Morse sequence, which is the fixed point of the substitution $0\rightarrow 01, 1\rightarrow 10$, has unbounded 1-gap $k$-binomial complexity for $k \geq 2$. Also,  we want to mention for a Sturmian sequence and $g\geq 1$, all of its long enough factors are always pairwise $g-gap k$-binomially inequivalent for any $k \geq 2$. 
Furthermore, for Fibonacci sequence see trees in~\cite{ref01,ref02}. 
\section{Literature Review}
The Fibonacci word, constructed via the substitution rules $0 \to 01$ and $1 \to 0$, is a well-studied infinite sequence in combinatorics on words. Its density properties—particularly the frequency of symbols—have been analyzed through multiple mathematical lenses. Furthermore see \cite{path2,path3,path4,path5,path6}.
\begin{definition}
The \emph{period length} of the Fibonacci sequence modulo~$p$, denoted $\pi(p)$, is the smallest integer $m \geq 1$ such that $F(n + m) \equiv F(n) \mod p$ for all $n \geq 0$.
The \emph{restricted period length} of the Fibonacci sequence modulo~$p$, denoted $\alpha(p)$, is the smallest integer $m \geq 1$ such that $F(m) \equiv 0 \mod p$.
\end{definition}

Let $L(n)_{n \geq 0}$ be the sequence of Lucas numbers, defined by $L(0) = 2$, $L(1) = 1$, and $L(n + 2) = L(n + 1) + L(n)$ for $n \geq 0$.

\begin{definition}
Let $p$ be a prime, and let $i \in \{0, 1, \dots, \pi(p) - 1\}$.
We say that $i$ is a \emph{Lucas zero} (with respect to $p$) if $L(i) \equiv 0 \mod p$ and a \emph{Lucas non-zero} if $L(i) \nequiv 0 \mod p$.
\end{definition}
\begin{definition}[Density of Symbols]\cite{path2}
The Fibonacci word exhibits a precise asymptotic density of $1$'s equal to $\frac{1}{\phi^2}$, where $\phi = \frac{1+\sqrt{5}}{2}$ is the golden ratio. This arises because the ratio of $0$'s to $1$'s converges to $\phi$. 
\end{definition}
Crucially, this density is uniform\footnote{See: \url{https://mathoverflow.net/questions/323614/is-the-density-of-1s-in-the-fibonacci-word-uniform}.}: for any position $c$ and interval length $m$, the proportion of $1$'s in the substring $ \omega_{c-m}^{c+m} $ converges uniformly to $\frac{1}{\phi^2}$ as $m \to \infty$. 
\begin{proposition}\cite{path3}
For any substring of length $n$,  
$$
\left| \frac{\text{number of } 1\text{'s}}{n} - \frac{1}{\phi^2} \right| \leq \frac{1}{n},
$$  
ensuring the density becomes increasingly homogeneous for large $n$.  
\end{proposition}
Then, we have: 
\begin{enumerate}
    \item \textbf{Critical Factorizations and Structural Properties } 
While not directly about symbol density, the Fibonacci word’s structural regularity influences its density behavior. For example, Fibonacci words longer than five characters have \textbf{exactly one critical factorization} a position where the local period equals the global period[3]. This uniqueness contrasts with palindromes, which have at least two critical points (See~\cite{path2, path4}. 
\item \textbf{Generalized Fibonacci Sequences and Density} Studies on generalized Fibonacci sequences (e.g., $(r, a)$-Fibonacci numbers) explore natural density in number-theoretic contexts. While these sequences have zero natural density in $\mathbb{Z}_{\geq 1}$, such results differ from symbol density in the Fibonacci word, highlighting distinct applications of ``density'' across mathematical domains (See~\cite{path2,path5}).
\end{enumerate}
\begin{definition}
The Fibonacci sequence $F(n)_{n \geq 0}$ is defined by the initial conditions $F(0) = 0$ and $F(1) = 1$ and the recurrence
\begin{equation}\label{recurrence}
	F(n + 2) = F(n + 1) + F(n)
\end{equation}
for $n \geq 0$.
It is well known that $F(n)_{n \geq 0}$ is periodic modulo~$m$.
For example, the Fibonacci sequence modulo~$7$ is
\[
	0, 1, 1, 2, 3, 5, 1, 6, 0, 6, 6, 5, 4, 2, 6, 1, 0, 1, 1, 2, 3, 5, 1, 6, \dots
\]
with period length $16$.
\end{definition}
\begin{definition}
Let $p$ be a prime.
The \emph{limiting density} of the Fibonacci sequence modulo powers of $p$ is
\[
	\dens(p) \colonequal \lim_{\lambda \to \infty} \frac{\lvert\{F(n) \bmod p^\lambda : n \geq 0\}\rvert}{p^\lambda}.
\]
\end{definition}
\begin{example}[Tribonacci word]
The first iterations of $\sigma_3$ on $1$ are
\begin{align*}
&\sigma_3(1)=12;\\
&\sigma_3^2(1)=\sigma_3(12)=\sigma_3(1)\sigma_3(2)=1213;\\
&\sigma_3^3(1)=1213121;\\
&\sigma_3^4(1)=1213121121312. \end{align*}
and for all $k$, $\sigma_3^k(1)$ is a prefix of $\mathbf w_3$, which is a infinite, uniformly recurrent word. The frequencies of the digits $1$, $2$ and $3$ in $\mathbf w_3$ are respectively $\tau^{-1}$, $\tau^{-2}$ and $\tau^{-3}$, where $\tau$ is the Tribonacci constant, namely the greatest positive solution of $\tau^3=\tau^2+\tau+1$. 
\end{example}

\begin{definition}\cite{path00}
     Let $\mathbf{w}=\left(a_{n}\right)_{n \geqslant 0} \in \Sigma^{\mathbb{N}}$. The arithmetic closure of $\mathbf{w}$ is the set

$$
A(\mathbf{w})=\left\{a_{i} a_{i+d} a_{i+2 d} \cdots a_{i+k d} \mid d \geqslant 1, k \geqslant 0\right\} .
$$
\end{definition}

The arithmetic complexity of $\mathbf{w}$ is the function $a_{\mathbf{w}}$ mapping $n$ to the number $\mathrm{a}_{\mathbf{w}}(n)$ of words with length $n$ in $A(\mathbf{w})$.

If $\mathrm{a}_{\mathbf{w}}(k)=(\# \Sigma)^{k}$, then the word $\mathbf{w}$ is said to have full arithmetic complexity. The following statement immediately follows from the definition:

\begin{proposition}~\label{p03}
Let $\mathbf{W} \in \Sigma^{\mathbb{N}}$ and $\# \Sigma=k$. Then for all $n \in \mathbb{N}$ we have

$$
1 \leqslant \mathrm{p}_{\mathbf{w}}(n) \leqslant \mathrm{a}_{\mathbf{w}}(n) \leqslant k^{n} .
$$
\end{proposition}
\subsection{Infinite square-free words}

The infinite word of Thue- Morse has square factors. In fact, the only forecourt-free words over two letters a and b are:$\{a,b,ab,aba,bab\}.$
On the opposite, there will be infinite square-free words over three letters. This will now be shown.
\begin{definition}~\label{d.3.1}
  let $A=\{a,b\}$, and $B=\{a,b,c\}$,  Define a morphism: 
\begin{align*}
    & \delta:B^* \rightarrow  A^*;\\
    & \delta(c)=a,\delta (b)=ab,\delta(a)=abb
\end{align*}
For any infinite word b on $B$, 
$\delta(b)=\delta(b_{0})\delta(b_{1})\dots\delta(b_{n})$
is an endless word on A that is clearly defined that begins with the letter $a$ 
\end{definition}
M. Lothaire in ~\cite{L15} (see also ~\cite{L5, K19}) Consider,
 $a=y_{0} y_{1}\dots y_{n} \dots, \quad where \quad y_{n}\in \{a,ab,abb\}=\delta(B) $ 

Since $bbb$ connects, each an in an is actually followed by no more than two b before a new $a$. The factorization is also special. Consequently, there is a special infinite word $b$ on $B$ such that $\delta(b)=a$ ~\cite{L15}.

Grimm, Uwe in 2001 in ~\cite{L16} introduced some definition are important about Infinite square-free words and bounds as: 
\begin{definition}[Square-Free Words]
    These are sequences that do not contain any consecutive repeated substrings (squares).
\end{definition}
\begin{definition}[Binary Square-Free Words]
    Limited to short sequences like $a,b, ab, ba, aba, bab$.
\end{definition}
\begin{definition}[Ternary Square-Free Words] ~\cite{L18,L30}
    Can be infinitely long, and the number of such words grows exponentially with their length. The language of ternary square-free words is 
\begin{equation*}
    \mathcal{A}=\bigcup_{n\ge 0}\mathcal{A}(n)\subset\Sigma^{\mathbb{N}_0}.
\end{equation*}
\end{definition}

Brandenburg in~\cite{L17} (see \cite{L31}) introduce a definition on bounds as 
\begin{proposition} [Bounds]~\cite{L17,L30} ~\label{t.num}
    The number of ternary square-free words of length $n$ is bounded by $$6 \cdot 1.032^n \leq s(n) \leq 6 \cdot 1.379^n$$ as shown by Brandenburg in 1983, for any $n\ge 3$.  in~\cite{L30} This bound can be systematically improved by
calculating $a(n)$ for as large values of $n$ as possible, from $a(90)=258\,615\,015\,792$, is
\begin{equation*}
s\le{43\,102\,502\,632}^{\frac{1}{88}}=1.320\,829\,\ldots
\end{equation*}
 The value $a(110)$ yields an improved upper bound of
\begin{equation*}
s\le {8\,416\,550\,317\,984}^{\frac{1}{108}}=1.317\,277\,\ldots
\end{equation*}
\end{proposition}
\subsection{Ergodic theory}
In mathematics, a sequence $\left(s_{n}\right)_{n \geqslant 0}$ of real numbers is said to be equidistributed or uniformly distributed on a non-degenerate interval $[a, b]$, if the proportion of terms that fall into a sub-interval is proportional to the length of this interval, i.e., if for any sub-interval $[c, d]$ of $[a, b]$ we have

$$
\lim _{n \rightarrow \infty} \frac{\#\left(\left\{s_{1}, \ldots, s_{n}\right\} \cap[c, d]\right)}{n}=\frac{d-c}{b-a} .
$$

The theory of uniform distribution modulo 1 deals with the distribution behavior of sequences of real numbers.

\begin{definition}
    A sequence $\left(a_{n}\right)_{n \geqslant 0}$ of real numbers is said to be equidistributed modulo 1 or uniformly distributed modulo 1 if the sequence of fractional parts of $\left(a_{n}\right)_{n \geqslant 0}$ is equidistributed in the interval $[0,1]$.
\end{definition}

\begin{theorem}[Weyl's criterion\cite{path01}]
 A sequence $\left(a_{n}\right)_{n \geqslant 0}$ is equidistributed modulo 1 if and only if for all non-zero integers $N$,

$$
\lim _{n \rightarrow \infty} \frac{1}{n} \sum_{j=1}^{n} \mathrm{e}^{2 \pi \mathrm{i} N a_{j}}=0.
$$
\end{theorem}
\begin{lemma}\cite{path02}~\label{pa01}
    The fractional part of the sequence $(\log (n!))_{n \geqslant 0}$ is dense in $[0,1]$.\\[0pt]
\end{lemma}
\begin{proposition}\label{pa02}
If $k$ is any positive integer having $m$ digits, there exists a positive integer $n$ such that the first $m$ digits of $n$ ! constitute the integer $k$.
\end{proposition}

Also we can state this result for any arbitrary base. Let $d_{m} \ldots d_{n}$ be a word over $\Sigma_{b}$ with $d_{m} \neq 0$. There exists $n$ such that the base $_{b}$ expansion of $n!$ starts with $d_{m} \ldots d_{n}$.
\subsection{Sturmian Morphisms}
\begin{definition}[The Number of Occurrences]
    If we have the element $a \in A$ and $w \in A^*$ (we write $A^*$ it's the set of all words on A) intends the number of iterations $a$ in $w$ by $|w|_a$ and denotes the number of occurrences of $b$ in $w$ by $|w|_b$.
If $w=a b a a b$, we can put: $|w|_a=3, |w|_b=2.$
\end{definition}
A morphism is a function that maps words (finite sequences) to other words while preserving the structure of the sequences. Sturmian morphisms specifically refer to the transformations that generate Sturmian sequences from simpler words.
\begin{definition}
 The following three morphisms should be defined:
\[
E\left\{\begin{array}{l}
	a \rightarrow b \\
	b \rightarrow a
\end{array} \varphi=\left\{\begin{array}{l}
	a \rightarrow a b \\
	b \rightarrow a
\end{array} \quad \tilde{\varphi}=\left\{\begin{array}{l}
	a \rightarrow b a \\
	b \rightarrow a
\end{array}\right.\right.\right.
\]
Where $\varphi$ is the Fibonacci morphism.   
\end{definition}
 A morphism $\psi=A^* \rightarrow A^*$ is Sturmian if the infinite word $\psi(x)$ is Sturmian for any Sturmian word $x$. F. Mignosi and P. Seebold ~\cite{21} mention to a morphism is $\psi$ is Sturmian if  $\psi\left\lbrace E,\varphi,\tilde{\varphi}\right\rbrace^*$ and only if it’s composed of $E,\varphi,\tilde{\varphi}$ in any number and order. Also, if $\psi(x)$ is a characteristic Sturmian word for every characteristic Sturmian word x, a morphism is standard.

\begin{definition}[Generating Structures]
Sturmian morphisms can be employed to generate infinite families of graphs. By applying these morphisms to base graphs, one can create complex structures that exhibit properties similar to those found in Sturmian sequences, such as balance and aperiodicity.
\end{definition}
\begin{definition}
Strings or sequences that contain subsets of characters that can create palindromic structures are referred to be scattered palindromic. Sporadic palindromic structures are more flexible than standard palindromes, which have tight symmetry requirements. For scattered palindromes, we can define them in terms of their indices. 
\end{definition}
\begin{lemma}
Let be a word $w$ of order $n$, A sequence $s_1, s_2, \ldots, s_k; \quad 1\leq k \leq n$ is scattered palindromic if there exists some mapping such that:
$$
s_i = s_j \quad \text{where } j = k + 1 - i \quad \text{for some } i,j.
$$
\end{lemma}
if we return to Proposition~\ref{t.num} we can notice that clearly, in next example we explain that as: 
\begin{example}
    Let be a word $w=aabb$, The scattered palindromic sub strings include: 
    \begin{enumerate}
        \item Single characters: $a, b$, 
        \item Pairs of identical characters: $aa, bb$, 
        \item Longer combinations that still maintain some palindromic properties: $aba$ (though not present in ``$aabb$'').
    \end{enumerate}
\end{example}

\begin{theorem}
	If we have $w$ any palindrome word, then it’s said that a word $v\in A^*$ is a subword of another word $x\in A^*$ if:	
	$v=a_1 a_2 \ldots a_n, a_i \in A, n \geq 0$ then we have also as: 
 $\exists \quad y_0, y_1, \ldots, \mathrm{y}_n \in \mathrm{A} ; \mathrm{x}=y_0 a_1 y_1 a_2 \ldots y_n a_n .$
	Therefore, we have: 
\begin{align*}
& |p(w)| \leq|s p(w)| ;\forall w \in \Sigma \\
& where \quad s p(w)=\sum_{t=1}^{|w|} s p_t(w) ; \quad \mathrm{t}: \text { length of word } w .\\
\end{align*}
\end{theorem}

\subsection{Statement of problem}
The infinite word $\mathfrak{F}=\mathfrak{F}_{b}:=\left({ }_{b} f_{n}\right)_{n \geqslant 0}$ is defined by concatenating non-negative base- $b \geqslant 2$ representation of the recursive $n$ !.\\
by concatenating base-10 representation of the recursive $n!$ :

$$
\mathfrak{F}:=\left(f_{n}\right)_{n \geqslant 0}=112624120720504040320 \cdots .
$$

What is the factor complexity of the $\mathfrak{F}$, i.e. $\mathrm{p}_{\mathfrak{F}}(k)$ ? What about arithmetic complexity, i.e., $\mathrm{a}_{\mathfrak{F}}(k)$ ? In fact, this problem can be easily generalized for any natural bases.\\

\section{Main Result}
Actually, the first time we presented the Fibonacci density concept according to the fundamental principles of density was at the 65th Congress of the Moscow Institute of Physics and Technology in 2021 and the results were published in~\cite{path00}. Sturmian words can take on a multitude of equivalent meanings and show a wide range of features; in particular, their palindromic or return word structures can help to distinguish them. All infinite words with exactly $n+1$ distinct subwords of length  $\mathrm{n}$. Glen et al. In~\cite{path7} mention for each $n \in \mathbb{N}$  belong to the family of Sturmian words. If $p(w, n)=n+1, \mathrm{n} \geq 0$ and only if, obviously $p(w, 1)=2$.

If all $n \geq 0$ of an endless word's subwords of length $\mathrm{n}$ are equal to $\mathrm{n}+1$, the word is said to be Sturmian.
The Fibonacci word sequence is generated as follows:

\( S(0) = a, \)
\( S(1) = ab, \)
\( S(2) = aba, \)
\( S(3) = abaab, \)
\( S(4) = abaababa, \)

Below in Figure ~\ref{fig2} representation of these subwords 
\begin{figure}[H]
    \centering
 \begin{tikzpicture}[scale=.9]
 \draw  (3,4)-- (4.44,5.04);
 \draw  (3,4)-- (4.52,3);
 \draw  (4.44,5.04)-- (6,6);
 \draw  (4.44,5.04)-- (6,4);
 \draw  (4.52,3)-- (6,3);
 \draw  (6,6)-- (8,6);
 \draw  (8,6)-- (10,6);
 \draw  (6,4)-- (8,4);
 \draw  (6,3)-- (8,4);
 \draw  (6,3)-- (8,2);
 \draw  (8,2)-- (10,2);
 \draw  (8,4)-- (10,5);
 \draw  (8,4)-- (10,4);
 \draw  (8,4)-- (10,3);
 \draw (2.72,4.6) node[anchor=north west] {$\sigma$};
 \draw (4.14,5.8) node[anchor=north west] {$a$};
 \draw (4.26,3.04) node[anchor=north west] {$b$};
 \draw (5.88,6.8) node[anchor=north west] {$aa$};
 \draw (5.82,4.72) node[anchor=north west] {$ab$};
 \draw (5.94,2.88) node[anchor=north west] {$ba$};
 \draw (8.02,6.8) node[anchor=north west] {$aab$};
 \draw (10.02,6.8) node[anchor=north west] {$aaba$};
 \draw (7.86,4.88) node[anchor=north west] {$bba$};
 \draw (7.7,2.7) node[anchor=north west] {$bab$};
 \draw (9.44,2.7) node[anchor=north west] {$baba$};
 \draw (9.66,3.7) node[anchor=north west] {$abab$};
 \draw (9.62,4.7) node[anchor=north west] {$baab$};
 \draw (9.74,5.7) node[anchor=north west] {$abaa$};
 \begin{scriptsize}
 \draw [fill=white] (3,4) circle (2.5pt);\draw [fill=white] (4.44,5.04) circle (2.5pt);\draw [fill=white] (4.52,3) circle (2.5pt);\draw [fill=white] (6,6) circle (2.5pt);\draw [fill=white] (6,4) circle (2.5pt);\draw [fill=white] (6,3) circle (2.5pt);\draw [fill=white] (8,6) circle (2.5pt);\draw [fill=white] (10,6) circle (2.5pt);\draw [fill=white] (8,4) circle (2.5pt);\draw [fill=white] (8,2) circle (2.5pt);\draw [fill=white] (10,2) circle (2.5pt);\draw [fill=white] (10,5) circle (2.5pt);\draw [fill=white] (10,4) circle (2.5pt);\draw [fill=white] (10,3) circle (2.5pt);\end{scriptsize}\end{tikzpicture}
 \caption{The subwords of the Fibonacci word.}
	\label{fig2}
\end{figure}
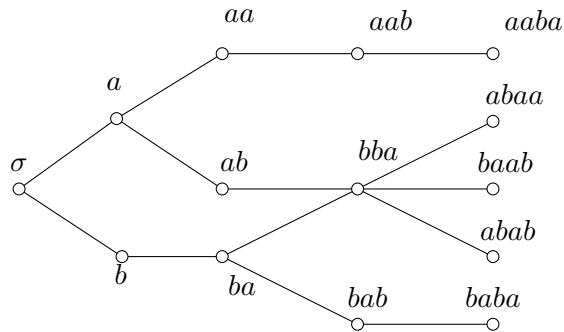
\begin{theorem}\label{main theorem}\cite{path6}
Let $p \neq 2$ be a prime, and define $e = \nu_p(F(p - \epsilon))$.
Let
\[
	N(p)
	= \left\lvert\left\{F(i) \bmod p^e : \text{$i$ is a Lucas non-zero}\right\}\right\rvert,
\]
and let $Z(p)$ be the number of Lucas zeros $i$ such that $F(i) \nequiv F(j) \mod p^e$ for all Lucas non-zeros $j$.
Then
\[
	\dens(p)
	= \frac{N(p)}{p^e}
	+ \frac{Z(p)}{2 p^{2 e - 1} (p + 1)}.
\]
\end{theorem}
\begin{theorem}
(i) The factor complexity of the infinite word $\mathfrak{F}_{b}$ is full.\\
(ii) The arithmetic complexity of the infinite word $\mathfrak{F}_{b}$ is full.
\end{theorem}
\begin{proof}
(i) Let the alphabet for base- $b$ is $\Sigma_{b}=\{0, \cdots, b-1\}$. Then $\mathfrak{F}_{b} \in \Sigma_{b}^{\mathbb{N}}$. Now we want to find

$$
k \mapsto \#\left\{f_{i} \cdots f_{i+k-1} \mid i \geqslant 0\right\} .
$$

By Lemma~\ref{pa01} and Proposition~\ref{pa02}, we claim that $\left({ }_{b} f_{n}\right)_{n \geqslant 0}$ is equidistributed modulo 1 . Then we have a same result such Proposition~\ref{pa02}, but for an arbitrary bases, i.e., there exists an $n$ such that the $b$-expansion of $n$ ! begins with these digits. On the other hand, each word $\mathbf{T} \in \Sigma_{b}^{+}$will appearance at least one position in $\mathfrak{F}_{b}$, i.e., $\mathfrak{F}_{b}[i \cdots j]=\mathbf{T}$, because there exist $\mathbf{s} \in \Sigma_{b}^{+}$such that it begins by $\mathbf{T}$.

Hence, $\mathcal{L}_{\mathfrak{F}_{b}}=\Sigma_{b}^{+}$, and $\mathfrak{F}_{b}$ is full factor complexity, i.e.,

$$
\mathrm{p}_{\mathfrak{F}_{b}}(k)=\left(\#\left(\Sigma_{b}\right)\right)^{k}=b^{k} .
$$

(ii) In the previous part we show that $\mathrm{p}_{\mathfrak{F}_{b}}(k)=b^{k}$. Now by Proposition~\ref{p03}, we can say that for all $k \in \mathbb{N}$

$$
b^{k} \leqslant \mathrm{a}_{\tilde{\mathrm{v}}_{b}}(k) \leqslant b^{k} .
$$

This inequality is true for all natural numbers $k$, this implies that $\mathrm{a}_{\mathfrak{F}_{b}}(k)=b^{k}$.
\end{proof}
\begin{lemma}~\label{3.1}
$(a_1, a_2,\dots, a_n); a_i \in A$ And all the words that we formed from A will be denoted by: $A^+=A^* -1$, and we call them the \textit{semi-group $A$}. 
\end{lemma}
 \begin{theorem}
The density of Fibonacci word a particular letter in the word using the following formula:
	$$\lim _{n\rightarrow \infty}D(n)=\lim _{n\rightarrow \infty}(\frac{F(n)}{\operatorname{F(n+1)}})=\varphi-1$$
where $\varphi$ is the golden ratio.
\end{theorem}

\section{Examples}
\begin{example}\cite{path6}
Let $p = 13$.
The period length is $\pi(13) = 28$, and the restricted period length is $\alpha(13) = 7$.
There are no Lucas zeros.
The set $\{F(0), \dots, F(27)\} \bmod 13$ is $\{0, 1, 2, 3, 5, 8, 10, 11, 12\}$.
Therefore $N(13) = 9$, and $\dens(13) = \frac{9}{13}$.
In fact, $\frac{9}{13}$ is the density of residues attained by the Fibonacci sequence modulo $13^\lambda$ for every $\lambda \geq 1$.
\end{example}

\begin{example}\cite{path6}
Let $p = 19$, for which $\pi(19) = 18 = \alpha(19)$.
The only Lucas zero is $9$.
The set
\[
	\{F(i) \bmod 19 : \text{$0 \leq i \leq 17$ and $i \neq 9$}\} = \{0, 1, 2, 3, 5, 8, 11, 13, 16, 17, 18\}
\]
has size $N(19) = 11$ and does not contain $(F(9) \bmod 19) = 15$.
Therefore $Z(19) = 1$, and $\dens(19) = \frac{11}{19} + \frac{1}{760} = \frac{441}{760}$.
This density is the limit of the decreasing sequence $1, \frac{12}{19}, \frac{210}{361}, \frac{3981}{6859}, \frac{75621}{130321}, \dots$ of densities of residues attained modulo $19^\lambda$ for $\lambda \geq 0$.
\end{example}

\begin{example}\cite{path6}
Let $p = 31$, for which $\pi(31) = 30 = \alpha(31)$.
The only Lucas zero is $15$, but $F(15) \equiv 21 = F(8) \mod 31$.
Therefore $Z(31) = 0$ and $\dens(31) = \frac{N(31)}{31} = \frac{19}{31}$.
\end{example}

\begin{example}\cite{path6}
Let $p = 7$, for which $\pi(7) = 16$ and $\alpha(7) = 8$.
The Lucas zeros are $4$ and $12$.
The set
\[
	\{F(i) \bmod 7 : \text{$0 \leq i \leq 15$ and $i \neq 4, 12$}\} = \{0, 1, 2, 5, 6\}
\]
has size $N(7) = 5$ and does not contain $(F(4) \bmod 7) = 3$ or $(F(12) \bmod 7) = 4$.
Therefore $Z(7) = 2$ and $\dens(7) = \frac{5}{7} + \frac{2}{112} = \frac{41}{56}$.

Level $\lambda$ in the tree contains the residues modulo~$7^\lambda$.
Dotted edges from a residue class $m$ on level $\lambda$ indicate an omitted full $7$-ary tree rooted at that vertex; that is, for every $\gamma \geq \lambda$ and for every integer $k \equiv m \mod 7^\lambda$ there exists $n \geq 0$ such that $F(n) \equiv k \mod 7^\gamma$.
\end{example}

\section{Conclusion}
The study establishes foundational results on the density and complexity properties of Fibonacci words, integrating combinatorial, algebraic, and dynamical perspectives. The Fibonacci word is confirmed as a \textbf{Sturmian word}, characterized by its minimal factor complexity $ p(n) = n + 1 $. This aligns with its generation via substitutions and hierarchical subword structure.  The asymptotic density of a specific letter in the Fibonacci word converges to $ \varphi - 1 \approx 0.618 $, where $ \varphi = \frac{1+\sqrt{5}}{2} $. This arises from the recursive ratio $ \lim_{n \to \infty} \frac{F(n)}{F(n+1)} $, directly tied to the golden ratio. Both factor complexity and arithmetic complexity of the infinite Fibonacci word $ \mathfrak{F}_b $ are proven to be full ($ b^k $ for base $ b $), demonstrating that all possible subwords and arithmetic progressions appear infinitely often.  
For primes $ p \neq 2 $, the density $ \text{dens}(p) $ is derived via counts of Fibonacci residues modulo $ p^e $ and Lucas zeros, combining combinatorial enumeration with modular arithmetic.  
 
These results deepen the understanding of Fibonacci words as prototypical Sturmian systems, with applications in symbolic dynamics, number theory, and coding. The explicit density formulas and complexity proofs provide tools for analyzing pattern distributions in substitutive sequences.

\bibliographystyle{plain}

\end{document}